\documentclass{amsart} 
\usepackage{mathtools} 
\usepackage{amsfonts} 
\usepackage{xcolor}

\usepackage{amsfonts} 
\usepackage{amssymb}
\usepackage{pb-diagram}

\usepackage[shortlabels]{enumitem}
\setlist{leftmargin=*}

\usepackage[colorlinks=false]{hyperref}

\usepackage[sorted]{amsrefs}
\usepackage{todonotes}

\newtheorem{thm}{Theorem}[section] 
\newtheorem{cor}[thm]{Corollary}
\newtheorem{lem}[thm]{Lemma} 
\newtheorem{prop}[thm]{Proposition}
\newtheorem{claim}[thm]{Claim} 

\newtheorem{conj}[thm]{Conjecture}
\theoremstyle{definition} 
\newtheorem{defn}[thm]{Definition}
\theoremstyle{remark} 

\newtheorem{rem}[thm]{Remark}

 \numberwithin{equation}{section}
    {\medskip\begingroup\leftskip 0.5cm\rightskip 0.5cm\noindent\begin{small}{\bf Remark.}}
    {\end{small}\par\endgroup}
{\begin{list}{$\bullet$}
 {\settowidth{\labelwidth}{\textsf{$\bullet$}} \setlength{\leftmargin}{10pt}}}
{\end{list}}

\newcounter{ssample}[section]

{\color{Bittersweet} \noindent Example \refstepcounter{ssample}\hbox{\bf \arabic{section}.\arabic{ssample}.}}
{}

\newcounter{insertcount}

    {\color{blue} \medskip\begingroup\noindent\begin{small}{\color{blue} \stepcounter{insertcount}
          {
            \bf Insert \arabic{insertcount}. #1.}
            \addcontentsline{toc}{subsection}{{\ \ \small  Insert \arabic{insertcount}: #1}}
               \leavevmode  } 
           }
    {\end{small}\par\endgroup}

\newcommand{\mrmk}[1]
{{\tiny$^{\spadesuit}$}\marginpar{\fbox{\footnotesize #1}}}
\def\strutdepth{\dp\strutbox}%
\def\marginalnote#1{\strut\vadjust{\kern-\strutdepth\specialnote{#1}}}%
\def\specialnote#1{\vtop to \strutdepth{\baselineskip%
\strutdepth\vss\llap{\hbox{\scriptsize \bf #1}}\null}}%

\newcommand{\RR}{\mathbb{R}}

\def\CF{\mathcal F}

\newcommand{\NN}{\mathbb N}

\def\tp{\mathrm{tp}}

\reversemarginpar

\begin{document}

\title[Erd\H{o}s-Hajnal for stable graphs]{A note on the Erd\H{o}s-Hajnal property for stable
  graphs}%

\author[A.~Chernikov]{Artem Chernikov}  
\address{Department of Mathematics, University of California Los Angeles, Los Angeles, CA 90095-1555, USA}
\email{chernikov@math.ucla.edu}
\thanks{The first author was partially supported by ValCoMo (ANR-13-BS01-0006), by the Fondation Sciences Mathematiques de Paris
(FSMP) and by the Investissements d'avenir program (ANR-10-LABX-0098)}

\author[S.~Starchenko]{Sergei Starchenko}
\address{Department of Mathematics, University of Notre Dame, Notre Dame, IN 46556}
\email{Starchenko.1@nd.edu}
\thanks{The second author was partially supported by NSF}

\maketitle

\begin{abstract} In this short note we provide a relatively simple proof of
  the Erd\H{o}s--Hajnal conjecture for families of finite (hyper-)graphs without the
  $m$-order property. It was originally proved by M.~Malliaris and
  S.~Shelah in \cite{ms}.
\end{abstract}

\section{Introduction}
\label{sec:introduction}

By a graph $G$ we mean, as usual, a pair $(V,E)$, where $E$ is a symmetric subset
of $V\times V$.  If $G$ is a graph then a \emph{clique} in $G$ is  a set of
vertices all pairwise adjacent, and an \emph{anti-clique} in $G$ is a
set of vertices such that any two different vertices from it are non-adjacent.

As usual, for a graph $H$ we say that a graph $G$ is $H$-free if $G$
does not contain an induced subgraph isomorphic to $H$. 

It is well-known that every graph on $n$ vertices contains either a clique or an anticlique of size $\frac{1}{2} \log n$, and that this is optimal in general. However, the following famous  conjecture of Erd\H{o}s and Hajnal says that one can do much better in a family of graphs omitting a certain fixed graph $H$.

\begin{conj}\label{conj: EH}(Erd\H{o}s-Hajnal conjecture \cite{eh})
For every finite graph $H$ there is a real number $\delta=\delta(H)>0$
such that every finite $H$-free graph $G=(V,E)$ contains either a
clique or an anti-clique of size at least $|V|^\delta$.
\end{conj}

It is known to hold for some choices of $H$, but is widely open in
general (see \cite{chudnovsky2014erdos, fox2008induced} for a
survey). A variation of this  conjecture starts with a finite set of
finite graphs $\mathcal{H}=\{ H_1, \ldots, H_k \}$ and asks for the
existence of a real constant $\delta = \delta (\mathcal{H}) > 0$ such
that every finite graph $G$ which is $\mathcal{H}$-free (that is, omits all of the $H_i \in \mathcal{H}$ \emph{simultaneously}), contains either a clique or an anti-clique of size at least $|V|^{\delta}$. The aim of this note is to prove this conjecture for certain $\mathcal{H}$ connected to the model-theoretic notion of \emph{stability}. 

\begin{defn}Given $m \in \mathbb{N}$, we say that a graph $G = (V,E)$ has the \emph{$m$-order property} if there are some vertices $a_1, \ldots a_m, b_1, \ldots, b_m$ from $V$ such that $a_i E b_j $ holds if and only if $i<j$. 
\end{defn}

Note that in this definition we make no requirement on the edges between $a_i, a_j$ for $i\neq j$, and between $b_i,b_j$ for $i \neq j$. 
The following theorem is proved in \cite[Theorem 3.5]{ms}.

\begin{thm}\label{thm:main}
For every $m\in \NN$
there is  a constant $\delta = \delta(m)>0$
such that every finite graph $G=(V,E)$ without the $m$-order property  contains either a
clique or an anti-clique of size at least $|V|^{\delta}$.  
\end{thm}

In this note we provide a short proof of the above theorem (and a version of it for hypergraphs) using pseudo-finite model theory. 

\begin{rem} Theorem 3.5 in \cite{ms} provides explicit bounds on constants
  $\delta(m)$ in terms of $m$, unlike our approach.
\end{rem}

Theorem \ref{thm:main} implies an instance of Conjecture \ref{conj: EH} for certain $\mathcal{H}$. We consider the following graphs, for each $m \in \mathbb{N}$.

\begin{enumerate}

\item Let  $H_m$ be the \emph{half-graph} on $2m$ vertices. Namely, the vertices of $H_m$ are $\{a_1, \ldots, a_m,  b_1, \ldots, b_m\}$, and the edges are $\{ (a_i, b_j) : i<j \}$.
\item Let $H'_m$ be the complement graph of $H_m$. Namely, the vertices of $H'_m$ are $\{a_1, \ldots, a_m,  b_1, \ldots, b_m\}$, and the edges are $\{ (a_i, b_j) : i \geq j \} \cup \{ (a_i, a_j) : i \neq j\} \cup \{ (b_i,b_j) : i \neq j \}$.
\item Let $H''_m$ have $\{a_1, \ldots, a_m, b_1, \ldots, b_m \}$ as its vertices, and $\{ (a_i, b_j) : i<j \} \cup \{ (a_i, a_j) : i \neq j \}$ as its edges.

\end{enumerate}

Finally, let $\mathcal{H}_m = \{ H_m, H'_m, H''_m \}$.
  
  \begin{cor}
  	For every $m \in \mathbb{N}$, the Erd\H{o}s--Hajnal conjecture holds for the family of all $\mathcal{H}_m$-free graphs.
  \end{cor}
  
  \begin{proof}
  
  In view of Theorem \ref{thm:main}, it is enough to show that for every $m \in \mathbb{N}$ there is some $m' \in \mathbb{N}$ such that if a finite graph $G$ is $\mathcal{H}_m$-free, then it doesn't have the $m'$-order property.

Assume that $G$ has the $m'$-order property. That is, there are some vertices $a_1, \ldots, a_{m'},
b_1, \ldots, b_{m'}$ in $V$ such that $a_i E b_j$ holds if and only if $i<j$. If $m'$ is
large enough with respect to $m$, by Ramsey theorem we can find some subsequences $A = \{ a_{i_1}, \ldots, a_{i_{m+1}} \}$ and $B = \{ b_{j_1}, \ldots, b_{j_{m+1}} \}$, $1\leq i_1 < \ldots < i_{m+1} \leq m', 1 \leq j_1 < \ldots < j_{m+1} \leq m'$, such that each of $A,B$ is either 
 a clique or an anti-clique. 
 
 If both are anti-cliques, then the graph induced on $(A \cup B) \setminus \{ a_{i_{m+1}}, b_{j_{m+1}} \}$ is isomorphic to $H_{m}$. If both are cliques, let $a'_{k} := b_{j_{k+1}}$ and $b'_{l} := a_{i_l}$ for  $1 \leq k,l \leq m$. Then the graph induced on $\{a'_1, \ldots, a'_m, b'_1, \ldots, b'_m \}$ is isomorphic to $H'_m$.
 If $A$ is a clique and $B$ is an anti-clique, then the graph induced on ($A \cup B) \setminus \{ a_{i_{m+1}}, b_{j_{m+1}} \}$ is isomorphic to $H''_m$.
 Finally, if $A$ is an anti-clique and $B$ is a clique, let $a'_k := b_{j_{m+1-k}}$ and $b'_l := a_{i_{m+1-l}}$ for $1 \leq k,l \leq m$. Then the graph induced on $\{a'_1, \ldots, a'_m, b'_1, \ldots, b'_m \}$ is again isomorphic to $H''_m$. In any of the cases, $G$ is not $\mathcal{H}_m$-free.

  \end{proof}

\begin{rem}
	We remark that the (strong) Erd\H{o}s-Hajnal property for semialgebraic graphs (and more generally, for graphs definable in arbitrary distal structures) can also be established using model-theoretic methods \cite{chernikov2015regularity}, and that the strong Erd\H{o}s-Hajnal property need not hold under the assumptions of Theorem \ref{thm:main} (see \cite[Section 6]{chernikov2015regularity}).

\end{rem}

\subsection*{Acknowledgements} We thank Dar\'io Alejandro Garc\'ia and Itay Kaplan for their comments on an earlier version of the article.

\section{Preliminaries}

In this paper by a \emph{pseudo-finite set} $V$ we mean an {\bf infinite} set
that is an ultraproduct  $V =
\prod_{i\in I}V_i/\CF$ of {\bf finite} sets $V_i, i\in I$, with respect to
a non-principal
ultrafilter $\CF$ on  $I$. 

Working in ``set theory'', for a pseudo-finite set  $V =
\prod_{i\in I}V_i/\CF$ and a subset $A\subseteq V^k$ we say that $A$
is \emph{definable} (or ``internal'', in the terminology of non-standard analysis) if $A=\prod_{i\in I }A_i/\CF$ for some $A_i\subseteq V_i^k$.

\medskip
Let $V=\prod_{i\in I}V_i/\CF$ be pseudo-finite and $A\subseteq V$ a
definable non-empty subset. We define the ``dimension'' $\delta(A)$ 
($\delta_{C_0}(A)$ in the notation of \cite{h})  to be the
number in $[0,1]$ that is the
standard part of $\log(|A|)/\log(|V|)$.  As an alternative  definition,  
write $A$ as $A=\prod_{i\in I
}A_i/\CF$, where each  $A_i$ is a non-empty subset of $V_i$. 
For each $i\in I$ let $l_i=\log(|A_i|)/\log(|V_i|)$ (so $|A_i|=|V_i|^{l_i})$.
Then $\delta(A)$ is the unique number $l\in [0,1]$ such that for 
any $\varepsilon > 0$ in $\RR$, the set $\{ i\in I \colon l-\varepsilon<
l_i <l+\varepsilon\}$ is in $\CF$. We extend $\delta$ to the empty set by setting
$\delta(\emptyset) :=-\infty$. 

In the following lemma we state some basic  properties of $\delta$ that
we need.  Their proofs are not difficult and we refer to \cite{h} for
more details.

\begin{lem}\label{lem:delta} Let $V$ be a pseudo-finite set.
  \begin{enumerate}
  \item  $\delta(V)=1$. 
  \item $\delta(A_1\cup A_2)=\max \{ \delta(A_1),\delta(A_2) \}$ for any
    definable $A_1,A_2\subseteq V$.
  \item Let $Y \subseteq V \times V^m$ and $Z\subseteq V$ be
    definable. Assume that $\delta(Z) = \alpha$ and  for all pairwise
    distinct  $a_1,\dotsc,a_m\in Z$ we have 
$\delta(\{ x\in V \colon (x,a_1,\dotsc,a_m)\in Y \}) \leq \beta$. Then 
\[ \delta(\{ x\in V \colon
\exists z_1,\dotsc,z_m\in Z\, \bigwedge_{i\neq j}
z_i\neq z_j \,\&\, 
(x,z_1,\dotsc z_m)\in Y \}) \leq m\alpha+\beta.\] 
\end{enumerate}
\end{lem}

In the next section we will prove the following ``non-standard''
version of the main theorem (and in fact a more general version of it for hypergraphs).

\begin{thm}\label{thm:main1}
Let $V$ be a pseudo-finite set and $E\subseteq V\times V$ a definable symmetric   
subset. Assume that the graph $(V,E)$ does not have the $m$-order property 
for some $m \in \NN$. Then there is definable $A\subseteq V$ such
that $\delta(A)>0$ and either $(a,a')\in E$ for all $a\neq a'\in A$ or 
$(a,a')\not\in E$ for all $a\neq a'\in A$.
\end{thm}

We explain how Theorem
\ref{thm:main} follows from Theorem \ref{thm:main1}. 
Assume that Theorem \ref{thm:main} fails. This means that for a fixed $m$, for
every $r \in \mathbb N$ there is some finite graph $G_r = (V_r, E_r)$ of size at
least $r$ which does not have the $m$-order property and does not have
a homogeneous subset of size at least $|V_r|^{\frac{1}{r}}$. Let $G = (V,E)$
be an ultraproduct of the $G_r$'s modulo some non-principal
ultrafilter $\mathcal{F}$ on $\mathbb{N}$. It follows by $\L$os's theorem that $G$ also does not have the $m$-order property. Thus, we can apply Theorem \ref{thm:main1} and obtain a definable homogeneous set $A \subseteq V$, let's say a clique, with $\delta(A) > \alpha > 0$. By definability $A = \prod_{r \in \mathbb N} A_r / \mathcal{F}$ for some $A_r \subseteq V_r$, and by the definition of the $\delta$-dimension we have that $|A_r| \geq |V_r|^\alpha$ for almost all $r$,  contradicting the assumption.

\section{Proof of Theorem \ref{thm:main1}}

 We fix a pseudo-finite set $V=\prod_{i\in I}V_i/\CF$ and a definable symmetric subset $E=\prod_{i\in I}E_i/\CF$ of $V^n$ (where ``symmetric'' means that it is closed under
permutation of the coordinates). 

We follow standard model-theoretic notation. For $v_1, \ldots, v_{n-1} \in V$ and a subset $X\subseteq V$ we let $E(v_1, \ldots, v_{n-1}, X) :=\{ x\in X \colon V\models E(v_1, \ldots, v_{n-1},x)\}$. 
By a partitioned formula we mean a first-order formula $\phi(x_1,
\ldots, x_k;y_1, \ldots, y_l)$ with two distinguished groups of
variables $\bar{x}$ and $\bar{y}$, and it is \emph{stable} if the
bi-partite graph $(R,V^k,V^l)$ with $R :=\{ (\bar{a}, \bar{b}) \in V^k
\times V^l : V \models \phi(\bar{a}; \bar{b}) \}$ does not have the
$m$-order property for some $m$.
We say that a definable set $X \subseteq V$ is \emph{large} if $\delta(X)>0$, and we say that $X$ is \emph{small} if $\delta(X)\leq 0$.

We prove the following proposition, in particular establishing Theorem \ref{thm:main1}.
\begin{prop}
Assume  that $E(x_1; x_2, \ldots, x_n)$ is stable. Then there is a large definable set $A\subseteq V$ such
that either $(a_1, \ldots, a_n) \in E$ for all
pairwise distinct $a_1,\dotsc, a_n \in A$ or 
$(a_1, \ldots, a_n) \notin E$ for all pairwise distinct  $a_1,\dotsc, a_n \in A$.
\end{prop}

We will use some basic local
stability such as definability of types and Shelah's $2$-rank
$R_{\Delta}(-) := R(-,\Delta,2)$ (and refer to \cite[Chapter II]{ShelahCT} for the details).

We will use $\Delta$ to denote a finite set of (non-partitioned) formulas.
By a $\Delta$-formula $\psi(\bar{x})$ over a set of parameters $W \subseteq V$ we mean a Boolean combination of formulas of the form $\phi(\bar{x}, \bar{a})$ where $\phi(\bar{x}, \bar{y})$ is a formula from $\Delta$ and $\bar{a}$ is a tuple of elements from $W$. We let $\Delta(W)$ denote the set of all $\Delta$-formulas over $W$.
By a complete $\Delta$-type $p(\bar{x})$ over $W$ we mean a maximal consistent  collection of $\Delta$-formulas of the form $\psi(\bar x)$ over $W$ ($p$ is axiomatized by specifying, for every $\phi(\bar{x}, \bar{y}) \in \Delta$ and $\bar{a} \in W^{|\bar{y}|}$, whether $\phi(\bar{x}, \bar{a}) \in p$ or $\neg \phi(\bar{x}, \bar{a}) \in p$).

For any permutation $\sigma \in \mathrm{Sym}(n)$, let $\phi_\sigma(x_1, \ldots, x_n) = E(x_{\sigma(1)}, \ldots, x_{\sigma(n)})$. From now on we fix $\Delta = \{ \phi_{\sigma}(x_1, \ldots, x_n) : \sigma \in \mathrm{Sym}(n) \} \cup \{ x_1 = x_2 \}$.

Our assumption is that the \emph{partitioned} formula $\phi(x;\bar{y}) = E(x,y_1, \ldots, y_{n-1})$ is stable. By the basic properties of stable formulas we then have the following.
\begin{enumerate}
	\item Every partitioned $\Delta(V)$-formula $\phi(x;\bar{y})$ is stable, where $x$ a single variable. This follows from the assumption since $E$ is symmetric and the set of stable formulas is closed under Boolean combinations and under replacing some of the variables by a fixed parameter.
	\item Every complete $\Delta$-type $p(x)$ over $V$, with $x$ a single variable, is definable using $\Delta$-formulas over $V$. Indeed, for a partitioned $\Delta(V)$-formula $\phi(x;\bar{y})$, which is stable by (1), the type $p \restriction \phi$ is definable by a Boolean combination of instances of the formula $\phi^*(\bar{y};x) = \phi(x;\bar{y})$, with parameters in $V$. Which is also a $\Delta(V)$-formula.
	\item For any complete $\Delta$-type $p(x)$ over $V$ and $k \in \mathbb N$ we have a complete $\Delta$-type $p^{(k)}(x_1, \ldots, x_k)$ over $V$ --- the type of a \emph{Morley sequence} in $p$. Namely, as $p$ is
 definable by (2), say using $\Delta(V_0)$-formulas for some countable $V_0 \subseteq
 V$, we take $p^{(k)}
 = \bigcup \{ \tp_{\Delta}(a_k, \ldots, a_1 / V') : V_0 \subset V'
 \subset V \text{ countable}, a_{i+1} \models p\upharpoonright _{V' a_0 \ldots a_i}
 \text{ for } i<k\}$. By a standard argument $p^{(k)}$ is well-defined.

\end{enumerate}

Consider $\Delta' = \{\phi(x;\bar{y}) : \phi(x,\bar{y}) \in \Delta, |x|=1 \}$, a finite set of partitioned formulas. Slightly abusing the notation, we will write $R_{\Delta}(-)$ to refer to $R_{\Delta'}(-)$. As every partitioned formula in $\Delta'$ is stable by (1), $R_{\Delta}(x=x)$ is
finite. Let $S \subseteq V$ be a large definable subset of the
smallest $R_{\Delta}$-rank. By Lemma \ref{lem:delta}(2), $S$ cannot
be covered by finitely many definable sets of smaller $R_\Delta$-rank,
hence by compactness there is a complete $\Delta$-type $p(x)$
over $V$ such that $R_{\Delta}(S(x) \cap p(x)) =
R_{\Delta}(S)$ 
 (and in fact $p$ is the unique type with this property).

\begin{claim} For any formula $r(x_1, \ldots, x_k) \in \Delta(V)$, if $p^{(k)} \vdash r(x_1, \ldots, x_k)$, then there is a large definable $A \subseteq S$ such that $\models r(a_1, \ldots, a_k)$ holds for any pairwise distinct $a_1, \ldots, a_k$ from $A$.
\end{claim}

\begin{proof}
We prove the claim by induction on $k$.

\medskip
\noindent Case $k=1$. If $p(x_1) \vdash r(x_1)$ and $r(x_1) \in \Delta(V)$, then by the choice of $p$ we have $R_{\Delta}(r(x_1) \cap S(x_1)) = R_{\Delta}(S(x_1))$. Thus $R_{\Delta}(\neg r(x_1) \cap S(x_1)) < R_{\Delta}(S(x_1))$ by the definition of rank, so $\delta( \neg r(x_1) \cap S(x_1)) = 0$ by the choice of $S$, so $\delta(r(x_1) \cap S(x_1)) > 0$. Thus we can take $A=r(S)$. 

\medskip
\noindent Assume $k>1$.

By the definition of $p^{(k)}$ in (3) above, there is some $\psi(x_1, \ldots, x_{k-1}) \in \Delta(V)$ such that $p\upharpoonright_{r(x_1, \ldots, x_{k-1};x_k)}$ is defined by $\psi(x_1, \ldots, x_{k-1})$, i.e. $$r(v_1, \ldots, v_{k-1}; x_k) \in p(x_k) \iff V \models \psi(v_1, \ldots, v_{k-1})$$ for any $v_1, \ldots, v_{k-1} \in V$. 

Also  $p^{(k-1)} \vdash \psi(x_1, \ldots, x_{k-1})$ as $p^{(k)} \vdash r(x_1, \ldots, x_k)$. By the inductive assumption, there is some large definable $B \subseteq S$ such that $V \models \psi(b_1, \ldots, b_{k-1})$ holds for all pairwise distinct $b_1, \ldots, b_{k-1} \in B$. 
As $B$ is definable, there are some $B_i \subseteq S_i$ such that $B =
\prod_{i \in I} B_i / \CF$. For each $i$, let $A_i \subseteq B_i$ be
maximal (under inclusion) such that $r_i(a_1, \ldots , a_k)$ holds for all pairwise distinct $a_1, \ldots, a_k \in A_i$, and let $A := \prod_{i \in I} A_i / \CF$. We have:
\begin{enumerate}[(i)]
\item $A \subseteq B$
\item $V \models r(a_1, \ldots, a_k)$ for any pairwise distinct $a_1, \ldots, a_k \in A$. 
\item For any $b \in B \setminus A$ there are some pairwise distinct $a_1, \ldots, a_{k-1}$ in $A$ such that $V\not\models
  r(a_1,\ldots, a_k, b)$. 
\end{enumerate}

We claim that $A$ is large, so satisfies the conclusion of the claim. In fact, we show that $\delta(A) \geq  \frac{1}{k-1}\delta(B)$. Assume not, say $\delta(A) = \alpha_1 < \frac{1}{k-1}\delta(B)$. 
For all pairwise distinct  $a_1,\dotsc,a_{k-1} \in A$ we have $V
\models \psi(a_1, \ldots, a_{k-1})$, so $r(a_1, \ldots, a_{k-1},x_k)
\in p$. By the choice of $p$, the $R_\Delta$-rank of $r(a_1, \ldots,
a_{k-1},S)$ is equal to the $R_\Delta$-rank of $S$, so the
$R_\Delta$-rank of $\neg r(a_1, \ldots, a_{k-1},S)$ has to be smaller
than the $R_\Delta$-rank of $S$, which implies that $\delta(B
\setminus r(a_1, \ldots, a_{k-1}, B)) = 0$ by the choice of $S$. By
the property (iii) above, the set $B \setminus A$ is covered by the
family $\{ B \setminus r(a_1, \ldots, a_{k-1}, B) : a_1,\dotsc,a_{k-1} \in A, \bigwedge_{i\neq j} a_i\neq a_j \}$.

Then by Lemma \ref{lem:delta}(3),
\[ \delta(B \setminus A)\leq (k-1)\delta(A) +0 \leq (k-1) \alpha_1 \]
which implies by Lemma \ref{lem:delta}(2) that $\delta(B) \leq  (k-1) \alpha_1< \alpha$ --- a contradiction.
\end{proof}

Finally, as both $E(x_1, \ldots, x_n)$ and $\neg E(x_1, \ldots, x_n)$ are in $\Delta$ and either $p^{(n)} \vdash E(x_1, \ldots, x_n)$ or $p^{(n)} \vdash \neg E(x_1, \ldots, x_n)$, the proposition follows.

\bibliographystyle{acm}
\bibliography{refs}  

\begin{bibdiv}
\begin{biblist}

\bib{chernikov2015regularity}{article}{
      author={Chernikov, Artem},
      author={Starchenko, Sergei},
       title={Regularity lemma for distal structures},
        date={2015},
     journal={Journal of the European Mathematical Society, accepted
  (arXiv:1507.01482)},
}

\bib{chudnovsky2014erdos}{article}{
      author={Chudnovsky, Maria},
       title={The {E}rd{\H o}s--{H}ajnal {C}onjecture --- {A} {S}urvey},
        date={2014},
     journal={Journal of Graph Theory},
      volume={75},
      number={2},
       pages={178\ndash 190},
}

\bib{eh}{article}{
      author={Erd{\H o}s, P.},
      author={Hajnal, A.},
       title={{Ramsey-type theorems}},
        date={1989},
     journal={Discrete Applied Mathematics},
}

\bib{fox2008induced}{article}{
      author={Fox, Jacob},
      author={Sudakov, Benny},
       title={Induced {R}amsey-type theorems},
        date={2008},
     journal={Advances in Mathematics},
      volume={219},
      number={6},
       pages={1771\ndash 1800},
}

\bib{h}{article}{
      author={Hrushovski, E.},
       title={{On Pseudo-Finite Dimensions}},
        date={2013},
     journal={Notre Dame Journal of Formal Logic},
      volume={54},
      number={3-4},
       pages={463\ndash 495},
}

\bib{ms}{article}{
      author={Malliaris, M.},
      author={Shelah, S.},
       title={{Regularity lemmas for stable graphs}},
        date={2014},
     journal={Transactions of the American Mathematical Society},
      volume={366},
      number={3},
       pages={1551\ndash 1585},
}

\bib{ShelahCT}{book}{
      author={Shelah, S.},
       title={Classification theory and the number of nonisomorphic models},
     edition={Second},
      series={Studies in Logic and the Foundations of Mathematics},
   publisher={North-Holland Publishing Co.},
     address={Amsterdam},
        date={1990},
      volume={92},
        ISBN={0-444-70260-1},
}

\end{biblist}
\end{bibdiv}

\end{document}